\documentclass[%
 aip,
 amsmath,amssymb,
 reprint,%
]{revtex4-1}

\usepackage{graphicx}
\graphicspath{{./figures/}}
\usepackage{dcolumn}
\usepackage{bm}
\usepackage[utf8]{inputenc}
\usepackage[T1]{fontenc}
\usepackage{mathptmx}
\usepackage{etoolbox}
\usepackage{amsthm}

\theoremstyle{plain}
\newtheorem{lemma}{Lemma}
\newtheorem{proposition}{Proposition}

\theoremstyle{definition}
\newtheorem{definition}{Definition}
\theoremstyle{plain}

\makeatletter
\def\@email#1#2{%
 \endgroup
 \patchcmd{\titleblock@produce}
  {\frontmatter@RRAPformat}
  {\frontmatter@RRAPformat{\produce@RRAP{*#1\href{mailto:#2}{#2}}}\frontmatter@RRAPformat}
 {}{}
}%
\makeatother

\newcommand{\What}{\widehat{W}}
\newcommand{\Wtrue}{W_{\mathrm{true}}}

\newcommand{\epsres}{\varepsilon_{\mathrm{res}}}
\newcommand{\epsmat}{\varepsilon_{\mathrm{mat}}}
\newcommand{\R}{\mathbb{R}}
\newcommand{\simplex}[1]{\Delta^{#1}}

\newcommand{\Epr}{\mathcal{E}_{p,r}}
\newcommand{\Fp}{\mathcal{F}_p}

\newcommand{\Mlift}{\mathcal{M}_p}

\begin{document}

\preprint{AIP/UNAH}

\title{Identifying Probability Localization Dynamics via Structured Stochastic Liftings}

\author{Fredy Vides}
\email{fredy.vides@unah.edu.hn}
\affiliation{
Department of Applied Mathematics, School of Mathematics and Computer Science, \\
Universidad Nacional Aut\'onoma de Honduras (UNAH)
}

\date{\today}

\begin{abstract}
This work develops a discrete-time framework for identifying probability
localization dynamics through finite stochastic representations adapted
jointly in space, time, memory, and state information. A compact
dynamically relevant set is localized by a finite measurable partition,
producing an observable probability state and a relational graph of
admissible transitions. Structured stochastic liftings derived from
Stochastically Structured Reservoir Computing (SSRC) give lossless
polynomial representations of the observable state, and stochastic delay
liftings add finite observable memory. These are distinguished from
\emph{dynamically informed state-space enrichment}: refinement of
observational fibers containing states with the same present observation
but different observable futures, yielding an exact
obstruction-to-closure criterion. Temporal coarse-graining is introduced
next, making the physical memory horizon explicit. A route-network toy
problem gives a minimal obstruction example, while four numerical
laboratories (rotational phase dynamics, the chaotic logistic map, the Van
der Pol oscillator, and a synthetic cyclic inventory system) show how
spatial scale, temporal scale, polynomial degree, and delay depth
interact. The logistic map isolates representation-induced memory in an
otherwise Markovian chaotic system, using its exact invariant law as an
ergodic benchmark and its zero-mass pseudospectrum to separate relaxation
from transient amplification. An exact rotational cycle calibrates
pseudospectra as a robustness diagnostic rather than a closure
certificate. The inventory example gives a closure-driven enrichment
procedure in which residence-age hazards trigger age-refined states that
improve predictive scores. These results motivate a \emph{minimal
adequate representation}: the least complex representation meeting
predictive, structural, and identifiability requirements.
\end{abstract}

\maketitle

\section{\label{sec:intro}Introduction}

Many dynamical systems are observed only through finite regions, categories,
or operational states. In that setting the useful state is often not a
pointwise vector in the original phase space, but the probability that the
system is localized in one of finitely many dynamically meaningful regions.
This motivates the study of \emph{probability localization dynamics}.

Let $X\subset\mathbb R^d$ be a compact dynamically relevant set and let
$\mathcal P_\delta=\{S_1,\ldots,S_m\}$ be a finite measurable localization
partition at spatial scale $\delta$. For a law $\mu_t$ on $X$, define
\begin{equation}
\Lambda_{\mathcal P_\delta}(\mu_t)
=
\begin{bmatrix}
\mu_t(S_1)&\cdots&\mu_t(S_m)
\end{bmatrix}^{\top}
=
p_t\in\Delta^{m-1}.
\label{eq:law-localization-intro}
\end{equation}
The vector $p_t$ may represent an ensemble distribution, an empirical
population, uncertainty in the initial condition or measurement, or a
stochastic law. For one deterministic trajectory with exact observation it
reduces to a one-hot state.

The central question is not only how to fit an update law for $p_t$, but how
to decide whether the chosen representation contains enough information for
the observable dynamics to close. This distinction is essential. A richer
polynomial feature map can increase nonlinear expressivity without adding
information. A delay window can add information distributed across
observable time. A refinement of the state space can add genuinely missing
mechanistic or exogenous information.

The SSRC architecture of Ref.~\onlinecite{ssrc2025ifac} is particularly
suited to this setting because it preserves stochasticity. It yields models
of the form
\begin{equation}
p_{t+1}
\approx
W\,\mathcal E_{p,r}(p_t,\ldots,p_{t-r}),
\end{equation}
where $W$ is column-stochastic and the lifted feature vector remains in a
simplex. When the linear block of the lifting is active, the observable state
or delay window is exactly recoverable from the lifted coordinates.

Four complementary ideas organize the framework. First, finite
localization geometry is treated explicitly at the law level and linked to a
relational graph. Second, instantaneous stochastic polynomial liftings are
separated from information-restoring state enrichment. Third, temporal
resolution $\alpha$ and delay depth $r$ are distinguished through the
physical memory horizon $r\alpha$. Fourth, state enrichment is made
diagnostic: residual structure, residence-age hazards, and candidate
exogenous variables can trigger local refinements of insufficient
observational fibers.

The resulting model family is naturally indexed by
\begin{equation}
(\delta,\alpha,p,r),
\label{eq:four-axes}
\end{equation}
together with graph support and optional state enrichment. The objective is
not maximal resolution along every axis, but a \emph{minimal adequate
representation} that meets a validation tolerance while controlling
complexity and finite-data conditioning.

The remainder of the paper is organized as follows. Section~\ref{sec:geometry} develops finite
localization geometry. Section~\ref{sec:ssrc} recalls structured stochastic
SSRC liftings. Sections~\ref{sec:prob-dynamics} and
\ref{sec:instant-lift} define the observable dynamics and instantaneous
liftings. Section~\ref{sec:enrichment} introduces dynamically informed
state-space enrichment. Section~\ref{sec:delay} treats observable memory and
Section~\ref{sec:temporal} temporal coarse-graining. Structured
identification and robustness are discussed in
Sections~\ref{sec:identification} and \ref{sec:spectral}. The numerical examples in Section~\ref{sec:example} use a compact
hierarchy: a route-network toy model illustrates the fiber obstruction,
rotational dynamics provide an exact phase and pseudospectral calibration,
the logistic map isolates representation-induced memory in a chaotic
discrete-time system, Van der Pol supplies a nonlinear spatial-temporal
localization study, and a synthetic inventory cycle provides the operational
enrichment example. Finally,
Section~\ref{sec:selection} develops adaptive representation selection and
Section~\ref{sec:discussion} summarizes the emerging theory.

\subsection{Related work and positioning}

The framework intersects several established lines of research, and it is
useful to state explicitly what is shared and what is different.

Memory induced by projection is the subject of the Mori--Zwanzig
formalism~\cite{zwanzig2001,chorin2000}: projecting a Markovian evolution
onto a reduced set of observables produces an exact reduced evolution law
with a memory kernel and a noise term. The present work shares the
diagnosis, since apparent memory in the examples arises from observation
rather than from the underlying state, but not the remedy. Instead of
modeling the memory kernel, the framework asks whether memory is
representational and, when it is, repairs the observation by targeted
state enrichment, or approximates the missing information by short
stochastic delay windows with explicit identifiability costs.

Koopman-operator methods and extended dynamic mode
decomposition~\cite{mezic2005,williams2015} also lift observables into
higher-dimensional spaces on which the dynamics acts approximately
linearly. The liftings used here differ in structure rather than in
spirit: the embeddings preserve the probability simplex exactly, the
readouts are constrained to be column-stochastic with graph support, and
the active linear block guarantees a left inverse, so the lifted
representation can never silently discard the observable state. General
EDMD dictionaries do not preserve stochasticity, and their coordinates
need not be recoverable.

The localization step itself is of Ulam type~\cite{ulam1960,dellnitz1999}:
a finite measurable partition induces a finite stochastic matrix. The
difference lies in what is estimated and how. Rather than discretizing a
known operator on a fixed partition, the framework identifies structured
stochastic representations from empirical probability trajectories and
treats partition geometry, temporal scale, polynomial degree, and delay
depth as jointly adaptable representation choices.

Markov state modeling in molecular dynamics faces the same non-Markovian
projection bias, and the approximation quality of aggregated Markov
models is well studied~\cite{sarich2010,deuflhard2005}. Markov state
models typically aggregate a fine simulated process into metastable
macrostates. Closure-driven enrichment proceeds in the opposite
direction: starting from a coarse observable representation, diagnostics
such as residence-age hazards or residual correlation with an exogenous
signal trigger local refinements, which also covers mechanistic and
exogenous enrichments that are not metastability driven.

Finally, the notion of a minimal adequate representation is a task-level,
finite-data counterpart of minimal predictive sufficient statistics in
computational mechanics~\cite{crutchfield1989,shalizi2001}. Causal states
are defined information-theoretically over the full past; a minimal
adequate representation is instead selected under explicit structural,
conditioning, and measurement-cost constraints, which is what makes it
operational for industrial and financial applications. On the
architectural side, SSRC itself descends from reservoir
computing~\cite{jaeger2004,ssrc2025ifac}, with the distinguishing feature
that stochastic structure is enforced during identification rather than
repaired afterwards.

\section{\label{sec:geometry}Finite localization geometry on a compact dynamical set}

Let $X\subset\mathbb R^d$ be compact. For any $\delta>0$, choose a finite
$\delta$-net $q_1,\ldots,q_m\in X$. A convenient deterministic partition is
obtained by nearest-representative assignment with fixed tie breaking:
\begin{equation}
S_j
=
\left\{
x\in X:
j=\min\operatorname*{arg\,min}_{1\le k\le m}
\|x-q_k\|
\right\}.
\label{eq:voronoi-partition}
\end{equation}
Then
\begin{equation}
X=\bigsqcup_{j=1}^{m}S_j,
\qquad
S_j\subset\overline B(q_j,\delta),
\qquad
\operatorname{diam}(S_j)\le 2\delta.
\end{equation}
The associated localization or quantization map is
\begin{equation}
Q_\delta(x)=q_j,
\qquad x\in S_j,
\end{equation}
with
\begin{equation}
\|x-Q_\delta(x)\|\le\delta.
\label{eq:quantization-error}
\end{equation}
No continuity of $Q_\delta$ is required.

\begin{definition}[Probability localization map]
For a probability law $\mu$ on $X$, define
\begin{equation}
\Lambda_{\mathcal P_\delta}(\mu)
=
\begin{bmatrix}
\mu(S_1)\\
\vdots\\
\mu(S_m)
\end{bmatrix}
\in\Delta^{m-1}.
\label{eq:law-map}
\end{equation}
\end{definition}

For deterministic dynamics $x_{t+1}=f(x_t)$,
\begin{equation}
p_{t+1}
=
\Lambda_{\mathcal P_\delta}(f_\#\mu_t),
\label{eq:pushforward-localization}
\end{equation}
where $f_\#$ denotes pushforward. More generally, $f_\#$ is replaced by the
law evolution induced by a stochastic or controlled system.

A structural graph $G_\delta=(V_\delta,E_\delta)$, with
$V_\delta=\{q_1,\ldots,q_m\}$, can be defined by
\begin{equation}
(i,j)\in E_\delta
\quad\Longleftrightarrow\quad
f(S_i)\cap S_j\neq\varnothing,
\label{eq:structural-edge}
\end{equation}
or by a positive transition probability for stochastic dynamics. With data,
one obtains an empirical graph by thresholding observed transition counts.
These structural and empirical graphs should be distinguished.

The graph is relational rather than necessarily physical: vertices may
represent geometric cells, phases, operational states, or regions of a
high-dimensional coupled system.

\subsection{Refinement and aggregation}

If $\mathcal P'$ refines $\mathcal P$, there is a deterministic
column-stochastic aggregation matrix $A$ such that
\begin{equation}
\Lambda_{\mathcal P}
=
A\Lambda_{\mathcal P'}.
\label{eq:partition-aggregation}
\end{equation}
Thus spatial enrichment by partition refinement has a canonical stochastic
projection back to the coarse representation.

\subsection{Exact law-level closure}

Let $\mathcal T$ denote the one-step evolution of laws and let
$\mathcal M$ be an admissible family of laws. The localization is exactly
closed on $\mathcal M$ if
\begin{equation}
\Lambda_{\mathcal P}(\mu)
=
\Lambda_{\mathcal P}(\nu)
\quad\Longrightarrow\quad
\Lambda_{\mathcal P}(\mathcal T\mu)
=
\Lambda_{\mathcal P}(\mathcal T\nu)
\label{eq:law-level-closure}
\end{equation}
for all $\mu,\nu\in\mathcal M$. This condition is the law-level version of
the fiber criterion developed later.

For an empirical ensemble of $N$ independent realizations,
\begin{equation}
\widehat p_j(t)
=
\frac1N\sum_{\ell=1}^{N}
\mathbf 1_{S_j}(x_t^{(\ell)}),
\end{equation}
and a direct union bound with Hoeffding's inequality~\cite{hoeffding1963} gives
\begin{equation}
\Pr\left(
\|\widehat p_t-p_t\|_\infty>\epsilon
\right)
\le
2m e^{-2N\epsilon^2}.
\label{eq:empirical-concentration}
\end{equation}
Hence spatial refinement improves geometric resolution but simultaneously
raises the dimension and sampling burden of the stochastic state.

\section{\label{sec:ssrc}Structured stochastic representations from SSRC}

The SSRC model of Ref.~\onlinecite{ssrc2025ifac} is recalled in the notation
needed below:
\begin{equation}
y(t)=W\,\eth_p(x(t))+e(t),
\label{eq:ssrc-model}
\end{equation}
where $x(t)$ and $y(t)$ are stochastic vectors, $W$ is
column-stochastic, and $e(t)$ is an identification residual.

A general unreduced stochastic $p$-embedding has the form
\begin{equation}
\widetilde{\eth}_{s,p}(x)
=
\frac{1}{m_s}
\begin{bmatrix}
s_1x\\
s_2x^{\otimes2}\\
\vdots\\
s_px^{\otimes p}\\
s_{p+1}
\end{bmatrix},
\qquad
s\in\{0,1\}^{p+1},
\label{eq:general-embedding}
\end{equation}
with
\begin{equation}
m_s=\sum_{j=1}^{p+1}s_j.
\end{equation}
The tensorial blocks preserve stochasticity because, for stochastic vectors
$x$ and $z$,
\begin{equation}
\mathbf{1}^{\top}(x\otimes z)
=
(\mathbf{1}^{\top}x)(\mathbf{1}^{\top}z)
=
1.
\end{equation}
Consequently every active tensor power $x^{\otimes k}$ is stochastic, and so
is the normalized block embedding.

Repeated words appear in the tensor powers. For example,
\begin{equation}
x_ix_j=x_jx_i
\end{equation}
although the two products occupy different coordinates in
$x^{\otimes2}$. Lemma~2 of Ref.~\onlinecite{ssrc2025ifac} introduces a sparse
matrix $R_{s,p}(n)$ that aggregates coordinates corresponding to the same
monomial word. The reduced embedding
\begin{equation}
\eth_{s,p,r}(x)
=
R_{s,p}(n)\widetilde{\eth}_{s,p}(x)
\label{eq:reduced-embedding}
\end{equation}
is stochastic and contains only non-redundant monomials.

The SSRC readout matrix is further restricted by a relational graph
$G_S=(V_S,E_S)$. If
\begin{equation}
B_S(m,n)
=
\{
e_{j,k}(m,n):(j,k)\in E_S
\},
\end{equation}
then the identified matrix satisfies
\begin{equation}
\What
\in
\bigl(\operatorname{span}B_S(m,n)\bigr)
\cap
\mathbb{S}_{m,n}(\R),
\label{eq:structured-W}
\end{equation}
and is obtained through the structured non-negative least-squares
procedure~\cite{lawson1974} of Ref.~\onlinecite{ssrc2025ifac}.

\subsection{\label{subsec:leftinverse}Recoverability of the observable state}

The feature of the SSRC embedding that is central here is not merely
stochasticity but recoverability.

Assume the linear block is active. Let $\alpha_1>0$ denote its coefficient
after normalization and reduction. Since degree-one monomials are distinct,
the reduction does not eliminate them. Therefore there exists a coordinate
selector $\Pi_1$ satisfying
\begin{equation}
\Pi_1\eth_p(x)=\alpha_1x.
\end{equation}
Define
\begin{equation}
\kappa_p(q)
=
\alpha_1^{-1}\Pi_1q.
\label{eq:kappa}
\end{equation}
Then
\begin{equation}
\kappa_p(\eth_p(x))=x
\label{eq:left-inverse}
\end{equation}
for every $x$ in the stochastic domain of the embedding.

\begin{proposition}[Injective stochastic lifting]
\label{prop:injective}
If the linear block of $\eth_p$ is active with nonzero coefficient, then
$\eth_p$ is injective on its stochastic domain and $\kappa_p$ is a left
inverse on $\operatorname{Im}(\eth_p)$.
\end{proposition}

\begin{proof}
If $\eth_p(x)=\eth_p(y)$, then applying $\kappa_p$ to both sides and using
Eq.~\eqref{eq:left-inverse} gives $x=y$.
\end{proof}

Thus the lifted state preserves the observable probability state exactly.

\section{\label{sec:prob-dynamics}Probability localization as a discrete-time dynamical system}

Let $G_X=(V_X,E_X)$ be a graph whose vertices represent observable regions,
classes, or relational states of a dynamical system. A probability
localization state is
\begin{equation}
p_t[i]
=
\Pr\{
\text{the system at time }t
\text{ is related to vertex }i
\}.
\label{eq:localization}
\end{equation}
The empirical version is obtained from a population or ensemble of observed
realizations.

The probability dynamics is written abstractly as
\begin{equation}
p_{t+1}
=
F(p_t)
\label{eq:prob-map}
\end{equation}
when the current probability vector is dynamically sufficient. The map $F$
need not be linear.

The simplest case is the stochastic linear model
\begin{equation}
F(p)=Wp,
\qquad
W\in\mathbb{S}_{n,n}(\R).
\label{eq:markov-case}
\end{equation}
This includes the usual finite-state Markov description and the matrices
obtained from finite state-space partitions. Classical transfer-matrix or
Ulam interpretations~\cite{ulam1960} can be attached to Eq.~\eqref{eq:markov-case} when
appropriate, but they are not required for the framework developed here.

The central issue is instead whether the chosen observable state closes the
dynamics. If no function $F$ exists such that Eq.~\eqref{eq:prob-map} holds
for all admissible hidden states compatible with the same $p_t$, then the
observable representation is dynamically insufficient.

\section{\label{sec:instant-lift}Instantaneous structured stochastic liftings}

Assume first that the observable probability vector is dynamically
sufficient. Consider an SSRC model
\begin{equation}
p_{t+1}
=
\What\eth_p(p_t).
\label{eq:ssrc-prob-map}
\end{equation}
Define
\begin{equation}
q_t
=
\eth_p(p_t),
\qquad
q_t\in
\Mlift
:=
\eth_p(\simplex{n-1}).
\label{eq:q-lift}
\end{equation}

The image $\Mlift$ is generally a proper nonlinear subset of the ambient
simplex. For a quadratic embedding, for example, its coordinates satisfy
algebraic relations induced by products $p_ip_j$.

The identified probability update on the original simplex is
\begin{equation}
\widehat F_p
=
\What\circ\eth_p.
\label{eq:Fhat}
\end{equation}
The corresponding closed lifted dynamics is obtained by re-embedding the
output:
\begin{equation}
q_{t+1}
=
\Fp(q_t),
\qquad
\Fp
:=
\eth_p\circ\What.
\label{eq:lifted-dynamics}
\end{equation}
For $p>1$, $\Fp$ is generally nonlinear.

\begin{proposition}[Conjugate lifted representation]
\label{prop:conjugacy}
Assume $\eth_p$ has the left inverse $\kappa_p$ of
Eq.~\eqref{eq:kappa}. Then, on $\Mlift$,
\begin{equation}
\Fp\circ\eth_p
=
\eth_p\circ\widehat F_p
\label{eq:commuting}
\end{equation}
and
\begin{equation}
\widehat F_p
=
\kappa_p\circ\Fp\circ\eth_p.
\label{eq:recover-F}
\end{equation}
Hence the identified probability dynamics and its lifted dynamics are
conjugate through the embedding restricted to its image.
\end{proposition}

\begin{proof}
For any $p\in\simplex{n-1}$,
\begin{equation}
\Fp(\eth_p(p))
=
\eth_p(\What\eth_p(p))
=
\eth_p(\widehat F_p(p)).
\end{equation}
Applying $\kappa_p$ yields Eq.~\eqref{eq:recover-F}.
\end{proof}

This yields a basic topological distinction within the framework: instantaneous
SSRC lifting is an injective change of representation, not a many-to-one
coarse observation.

\subsection{Quadratic lifting}

A useful example is
\begin{equation}
\eth_2(p)
=
\begin{bmatrix}
\alpha p\\
(1-\alpha)(p\otimes p)
\end{bmatrix},
\qquad
0<\alpha<1,
\label{eq:eth2}
\end{equation}
for which
\begin{equation}
\kappa_2(q)
=
\alpha^{-1}q_{1:n}.
\end{equation}
If
\begin{equation}
W_{\mathrm{emb}}
=
\begin{bmatrix}
W_1 & W_2
\end{bmatrix},
\end{equation}
then
\begin{equation}
p_{t+1}
=
\alpha W_1p_t
+
(1-\alpha)W_2(p_t\otimes p_t).
\label{eq:quadratic-dynamics}
\end{equation}
Thus a linear stochastic readout in the lifted coordinates induces a
nonlinear probability dynamics on the original simplex.

\section{\label{sec:enrichment}Dynamically informed state-space enrichment}

An instantaneous lifting cannot create information that is absent from the
observable state. This motivates a separate operation: refinement of the
state representation itself.

Let
\begin{equation}
\rho_{t+1}=G(\rho_t),
\qquad
p_t=H\rho_t,
\label{eq:fine-observation}
\end{equation}
where $\rho_t$ is an enriched state and $H$ is a stochastic observation or
aggregation map. The observational fiber through $\rho$ is
\begin{equation}
[\rho]_H
=
\{\rho':H\rho'=H\rho\}.
\end{equation}

\begin{proposition}[Obstruction to instantaneous closure]
\label{prop:no-closure}
If there exist $\rho^{(1)}$ and $\rho^{(2)}$ such that
\begin{equation}
H\rho^{(1)}=H\rho^{(2)}
\end{equation}
but
\begin{equation}
HG(\rho^{(1)})\neq HG(\rho^{(2)}),
\end{equation}
then there is no deterministic map
$F$ on the coarse observable state satisfying
\begin{equation}
HG=FH
\end{equation}
globally.
\end{proposition}

\begin{proof}
If $HG=FH$, then
$HG(\rho^{(1)})=F(H\rho^{(1)})=F(H\rho^{(2)})=HG(\rho^{(2)})$, a
contradiction.
\end{proof}

This criterion separates \emph{representational richness} from
\emph{informational sufficiency}. If $\eth_p$ is injective on the simplex,
then
\begin{equation}
H\rho^{(1)}=H\rho^{(2)}
\quad\Longrightarrow\quad
\eth_p(H\rho^{(1)})=\eth_p(H\rho^{(2)}).
\end{equation}
Increasing polynomial degree therefore cannot repair an information-losing
observation by itself.

\subsection{Enrichment as fiber refinement}

Let $\xi=\Xi(\rho)$ be an additional state variable and define
\begin{equation}
\widetilde H(\rho)
=
\big(H\rho,\Xi(\rho)\big).
\label{eq:enriched-observation}
\end{equation}
This construction is termed a \emph{dynamically informed state-space
enrichment} when $\widetilde H$ refines the original observational fibers
and reduces the variation of future observables within them.

For a prediction horizon $K$, the enrichment is $K$-step predictively
sufficient on a set $\mathcal A$ if
\begin{equation}
\widetilde H(\rho^{(1)})
=
\widetilde H(\rho^{(2)})
\end{equation}
implies
\begin{equation}
HG^k(\rho^{(1)})
=
HG^k(\rho^{(2)}),
\qquad
k=1,\ldots,K,
\label{eq:K-sufficiency}
\end{equation}
for $\rho^{(1)},\rho^{(2)}\in\mathcal A$. Exact reconstruction of the full
fine state is not required; only predictive distinctions need to be
retained.

Several enrichments used in this work fit this definition:
\begin{enumerate}
\item \emph{spatial refinement}, where a coarse localization cell is split;
\item \emph{phase enrichment}, where projected phases are separated;
\item \emph{transition-progress enrichment}, such as pending route edges or
order-transit age;
\item \emph{residence-age enrichment}, where elapsed time in a coarse state
is included;
\item \emph{exogenous-state enrichment}, where a forcing variable or its
phase is appended.
\end{enumerate}

\subsection{Closure-driven enrichment diagnostics}

The enrichment need not be specified entirely in advance. Let
\begin{equation}
e_t
=
p_{t+1}-\widehat F(\mathcal R_t)
\end{equation}
be the residual under the current representation $\mathcal R_t$.
Structured dependence of $e_t$ on candidate variables suggests directions
for refinement. Dependence on past observations suggests delay lifting;
dependence on an observed exogenous signal suggests appending that signal;
localized residual structure suggests spatial refinement.

A particularly useful trigger is residence-age dependence~\cite{kalbfleisch2002}. For a coarse
state $j$, define
\begin{equation}
h_j(a)
=
\Pr(s_{t+1}\neq j\mid s_t=j,A_t=a),
\label{eq:hazard}
\end{equation}
where $A_t$ is elapsed residence age. Geometric residence gives an
approximately constant hazard. Age dependence is quantified by
\begin{equation}
D_{\rm age}(j)
=
\frac{
\sum_a n_{j,a}
\left(
\widehat h_j(a)-\overline h_j
\right)^2
}{
\sum_a n_{j,a}
},
\label{eq:hazard-score}
\end{equation}
where $n_{j,a}$ is the number at risk. A constant-hazard null model can be
used to calibrate a data-driven trigger. Flagged states are then refined only
locally, for example
\begin{equation}
j
\longrightarrow
(j,a=1),(j,a=2),\ldots.
\end{equation}
The refinement is retained only if it improves held-out predictive closure
enough to justify the added complexity.

The resulting closure-driven enrichment loop is
\begin{equation}
\begin{aligned}
\text{fit}&\to\text{diagnose}\to\text{propose refinement}\\
&\to\text{validate}\to\text{accept/reject}.
\end{aligned}
\label{eq:enrichment-loop}
\end{equation}
The statistical procedure detects hidden progress; domain knowledge supplies
its mechanistic interpretation.

\section{\label{sec:delay}Structured stochastic liftings with finite memory}

Between instantaneous lifting and direct state-space enrichment lies a third
possibility: reconstruct missing predictive information from a finite
history of observable probability vectors, in the same reconstructive
spirit as delay-coordinate embedding of a deterministic
trajectory~\cite{takens1981}, but applied here to probability vectors
rather than to a single scalar observable.

Let
\begin{equation}
p_t,p_{t-1},\ldots,p_{t-r}\in\Delta^{d-1}
\end{equation}
and choose
\begin{equation}
v=(v_0,\ldots,v_r)^\top\in\Delta^r.
\end{equation}
Define
\begin{equation}
\Psi_{v,r}(p_t,\ldots,p_{t-r})
=
\begin{bmatrix}
v_0p_t\\
v_1p_{t-1}\\
\vdots\\
v_rp_{t-r}
\end{bmatrix}
\in\Delta^{d(r+1)-1}.
\label{eq:delay-block}
\end{equation}
The polynomial stochastic delay lifting is
\begin{equation}
\Epr
=
R_{p,r}\eth_p\circ\Psi_{v,r}.
\label{eq:Epr}
\end{equation}

\begin{proposition}[Recoverable stochastic delay lifting]
\label{prop:delay-left-inverse}
If the linear block of $\eth_p$ is active and $v_\ell>0$ for all $\ell$,
there exists a left inverse $\kappa_{p,r}$ on
$\operatorname{Im}(\Epr)$ such that
\begin{equation}
\kappa_{p,r}
\left(
\Epr(p_t,\ldots,p_{t-r})
\right)
=
(p_t,\ldots,p_{t-r}).
\end{equation}
\end{proposition}

The proof is immediate by selecting the first-order coordinates and
rescaling the delay blocks.

An identified update is
\begin{equation}
p_{t+1}
=
\widehat W_{p,r}
\Epr(p_t,\ldots,p_{t-r}).
\label{eq:delay-readout}
\end{equation}
The resulting closed lifted dynamics is obtained by reconstructing the
history, applying the readout, shifting the window, and lifting again.

Polynomial degree and delay depth have distinct roles:
\begin{equation}
p
\quad\text{enriches the instantaneous representation,}
\end{equation}
while
\begin{equation}
r
\quad\text{adds observable history.}
\end{equation}
They may interact, but they are not interchangeable. If two hidden states
have the same instantaneous observation, no injective instantaneous lifting
can separate them. A short delay may supply the missing distinction, after
which a nonlinear lifting can represent the resulting history more
efficiently.

The distinction is visible in the nonlinear oscillator laboratory: a
quadratic model with one delay outperforms all tested linear models through
$r=4$, but a sufficiently long linear history eventually performs better.
The observation supports memory compression by nonlinear richness, not
replacement of missing information by polynomial degree.

\section{\label{sec:temporal}Temporal coarse-graining and physical memory horizon}

Spatial resolution is only one discretization scale. Let $\alpha\ge1$ be an
integer temporal sampling factor and define
\begin{equation}
p^{(\alpha)}_k
=
p_{k\alpha}.
\label{eq:temporal-coarse}
\end{equation}
For an exact linear Markov model $p_{t+1}=Wp_t$,
\begin{equation}
p^{(\alpha)}_{k+1}
=
W^\alpha p^{(\alpha)}_k.
\end{equation}
The effective graph support must therefore be rebuilt at each temporal scale
using $\operatorname{supp}(W^\alpha)$ or empirical $\alpha$-step
transitions.

The parameters $\alpha$ and $r$ encode different quantities:
\begin{align}
\alpha&=\text{sampling interval},\\
r&=\text{number of remembered observations}.
\end{align}
and the physical memory horizon is
\begin{equation}
L_{\rm mem}=r\alpha.
\label{eq:physical-memory}
\end{equation}
This prevents a comparison such as $r=4$ at daily sampling and $r=4$ at
weekly sampling from being interpreted as equal memory.

Residence statistics provide candidate temporal scales. If
$\tau_{\rm dwell}$ denotes a residence episode, useful $\alpha$ values can
be drawn from its median, mean, quantiles, or from state-persistence
functions. However, large $\alpha$ can skip dynamically relevant
intermediate regions. A direct empirical diagnostic is the microscopic
crossing count
\begin{equation}
C_\alpha(t)
=
\sum_{k=0}^{\alpha-1}
\mathbf 1\{s_{t+k+1}\neq s_{t+k}\}.
\label{eq:crossing-count}
\end{equation}
Then
\begin{equation}
\Pr(C_\alpha=0)
\end{equation}
measures temporal redundancy and
\begin{equation}
\Pr(C_\alpha\ge2)
\end{equation}
measures multi-transition skipping. A practical temporal scale should
balance the two.

The Van der Pol experiments show that residence times vary with spatial
resolution, so $\alpha$ should be selected conditionally on $\delta$. This
motivates a coupled spatial-temporal search rather than independent tuning.

\section{\label{sec:identification}Structured identification and finite-data uncertainty}

All representations considered above can be written in a common form.
Let
\begin{equation}
\Phi_t
\end{equation}
denote either the original probability vector, an instantaneous lifting, or
a reduced delay lifting:
\begin{equation}
\Phi_t
\in
\left\{
p_t,\;
\eth_p(p_t),\;
\Epr(p_t,\ldots,p_{t-r})
\right\}.
\end{equation}
Construct
\begin{equation}
Y=
\begin{bmatrix}
\Phi_0 & \cdots & \Phi_{T-1}
\end{bmatrix},
\qquad
Y'=
\begin{bmatrix}
p_1 & \cdots & p_T
\end{bmatrix}.
\end{equation}
The SSRC estimator solves
\begin{equation}
\What
=
\arg\min_{W\in\mathcal{S}}
\|WY-Y'\|_F,
\label{eq:identification}
\end{equation}
where $\mathcal{S}$ encodes non-negativity, column-stochasticity, and any
graph-induced support constraints.

Define
\begin{equation}
\epsres
=
\|\What Y-Y'\|_F.
\label{eq:residual}
\end{equation}

For a representation-level model
\begin{equation}
Y'
=
\Wtrue Y+E,
\label{eq:data-model}
\end{equation}
the following finite-data estimate is obtained.

\begin{lemma}[Residual-to-matrix perturbation bound]
\label{lem:matrix-bound}
Suppose $Y$ has full row rank. Then
\begin{equation}
\|\What-\Wtrue\|_2
\leq
\frac{
\|\What Y-Y'\|_2+\|E\|_2
}{
\sigma_{\min}(Y)
}
\leq
\frac{
\epsres+\|E\|_2
}{
\sigma_{\min}(Y)
}.
\label{eq:matrix-bound}
\end{equation}
In the noise-free case $E=0$,
\begin{equation}
\|\What-\Wtrue\|_2
\leq
\epsmat
:=
\frac{\epsres}{\sigma_{\min}(Y)}.
\label{eq:epsmat}
\end{equation}
\end{lemma}

\begin{proof}
From Eq.~\eqref{eq:data-model},
\begin{equation}
(\What-\Wtrue)Y
=
(\What Y-Y')+E.
\end{equation}
Since $Y$ has full row rank,
\begin{equation}
\What-\Wtrue
=
\left[
(\What Y-Y')+E
\right]Y^\dagger.
\end{equation}
Taking spectral norms and using
$\|Y^\dagger\|_2=\sigma_{\min}(Y)^{-1}$ gives the result.
\end{proof}

This formulation makes the role of the lifting explicit. Increasing $p$ or
$r$ can reduce approximation error while simultaneously worsening
$\sigma_{\min}(Y)$. A richer representation is therefore useful only if the
available empirical probability trajectories sufficiently excite its lifted
coordinates.

\section{\label{sec:spectral}Pseudospectral robustness for square stochastic representations}

The broader lifting viewpoint requires care with spectral terminology.
For $p>1$ or $r>0$, the identified readout
\begin{equation}
\widehat W_{p,r}:
\mathbb R^{N_{p,r}}
\rightarrow
\mathbb R^n
\end{equation}
is generally rectangular, while the closed lifted dynamics
$\mathcal F_{p,r}$ is nonlinear. Eigenvalue and pseudospectral diagnostics
should therefore not be applied directly to $\widehat W_{p,r}$ as though it
were a square linear evolution matrix.

For a square matrix $A$, the $\epsilon$-pseudospectrum is defined as
\begin{equation}
\Lambda_{\epsilon}(A)
=
\left\{
z\in\mathbb C:
\sigma_{\min}(zI-A)\le\epsilon
\right\},
\label{eq:pseudospectrum}
\end{equation}
equivalently the set of eigenvalues of matrices $A+E$ with
$\|E\|_2\le\epsilon$; see Ref.~\onlinecite{trefethen2005spectra}.
For a simple eigenvalue $\lambda_i$, with left and right eigenvectors
$w_i$ and $v_i$ normalized by $w_i^\ast v_i=1$, the first-order
sensitivity scale is
\begin{equation}
|\delta\lambda_i|
\lesssim
\|w_i\|_2\|v_i\|_2\,\epsilon.
\label{eq:eigen-sensitivity}
\end{equation}

Two distinct uses are relevant here. First, when the matrix uncertainty
scale is estimated from the identification problem, for example by
$\epsilon=\epsilon_{\rm mat}$, the pseudospectrum describes spectral
uncertainty induced by finite-data matrix error. Second, at a common fixed
$\epsilon$, pseudospectra can compare the robustness of two alternative
square representations. This second use is diagnostic of representation
fragility, not a proof of missing information.

To make this distinction explicit, define the pseudospectral radius
\begin{equation}
r_\epsilon(A)
=
\sup\{|z|:z\in\Lambda_\epsilon(A)\}.
\end{equation}
For a normal matrix,
\begin{equation}
\Lambda_\epsilon(A)
=
\bigcup_{\lambda\in\sigma(A)}
\overline B(\lambda,\epsilon),
\end{equation}
so no excess inflation occurs beyond the $\epsilon$-neighborhood of the
spectrum. A strongly non-normal representation may display much larger
resolvent growth and transient sensitivity even when its eigenvalues appear
benign.

Accordingly, the logic used in this work is
\begin{equation}
\boxed{
\begin{aligned}
\text{closure diagnosis}
&\to
\text{representation/enrichment test}\\
&\to
\text{pseudospectral robustness}.
\end{aligned}
}
\label{eq:pseudo-logic}
\end{equation}
A large pseudospectrum does not imply a closure obstruction, and a
non-closed representation need not always have a dramatic pseudospectrum.

Square stochastic models arise naturally for order-one probability
dynamics and for explicitly enriched Markov states. For linear delay models,
a square companion representation can also be formed on the stacked delay
state. For nonlinear lifted dynamics, the corresponding local objects are
Jacobians such as
\begin{equation}
D\mathcal F_p(q)
\quad\text{or}\quad
D\mathcal F_{p,r}(q),
\end{equation}
and products of such Jacobians along trajectories. A full cocycle-level
analysis is beyond the scope of this paper.

\section{\label{sec:example}Illustrative examples}

The examples are organized around a common question: which observational
fibers are dynamically sufficient, and which representation change repairs
the detected defect?

\subsection{Route localization and hidden transition commitment}

Consider the route graph
\begin{equation}
O\to\{A,B\},
\qquad
A\to\{C,E\},
\qquad
B\to E,
\qquad
\{C,E\}\to D.
\end{equation}
A coarse state records only the associated vertex. A fine state also
distinguishes pending transitions such as
$\mathrm{pending}(A,C)$ and $\mathrm{pending}(A,E)$. Thus
\begin{equation}
He_{\mathrm{node}(A)}
=
He_{\mathrm{pending}(A,C)},
\end{equation}
while their projected futures differ. Enriching by transition commitment
restores closure in the synthetic fine model. This is the prototypical
\emph{transition-progress enrichment}.

\begin{figure*}[t]
\centering
\includegraphics[width=0.96\textwidth]{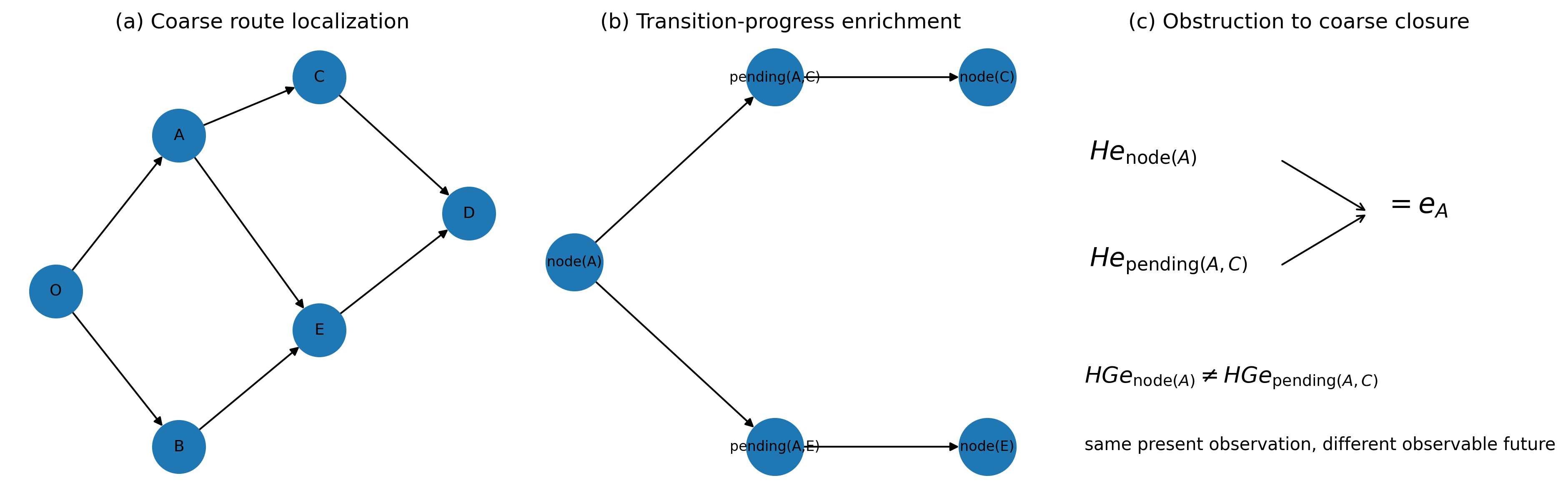}
\caption{\label{fig:route-enrichment}
Dynamically informed refinement of an observational fiber in the route toy
problem. (a) The coarse graph records only route vertices. (b) The enriched
state space separates transition commitment through pending-edge states.
(c) States can share the same present coarse observation while having
different projected futures, which obstructs a globally closed
instantaneous coarse map.}
\end{figure*}

\subsection{Rotational dynamics: phase, memory, and temporal scale}

For
\begin{equation}
x_{t+1}=A(\theta)x_t,
\qquad
A(\theta)=
\begin{bmatrix}
\cos\theta&-\sin\theta\\
\sin\theta&\cos\theta
\end{bmatrix},
\end{equation}
the scalar observation $x_1(t)$ identifies states
$(x_1,+x_2)$ and $(x_1,-x_2)$ that generally have different futures.
Mechanistic enrichment restores the missing phase coordinate, while one
delay is sufficient for the exact scalar recurrence
\begin{equation}
x_1(t+1)=2\cos\theta\,x_1(t)-x_1(t-1).
\end{equation}

A finite phase partition gives an even closer probability-localization
example. Coarse labels may merge different fine phases, producing
\begin{equation}
He_i=He_j,
\qquad
HP_Ne_i\neq HP_Ne_j.
\end{equation}
Finite histories can separate the phases for irregular partitions, whereas
persistent symmetries can remain ambiguous for every delay. This shows that
memory can reconstruct hidden state only when accumulated observations
separate the relevant fibers.

A random clock
\begin{equation}
W_\eta=(1-\eta)I+\eta P_N
\end{equation}
creates geometric dwell times and makes temporal coarse-graining explicit.
The obstruction magnitude is scaled by $\eta$ but is not removed by slower
dynamics.

The phase-cycle model also provides an exact pseudospectral calibration.
The fine permutation $P_N$ is unitary and therefore normal, so its
$\epsilon$-pseudospectrum is exactly the union of $\epsilon$-disks around
the $N$th roots of unity. The random-clock matrix
$W_\eta=(1-\eta)I+\eta P_N$ is normal as well. Hence residence time by
itself does not imply pseudo\-spectral fragility. By contrast, a square
coarse Markov model fitted after several phases are merged is forced to
represent a non-closed quotient by a single transition matrix. Its
pseudospectrum is used only as a complementary robustness diagnostic.

\begin{figure*}[t]
\centering
\includegraphics[width=0.96\textwidth]{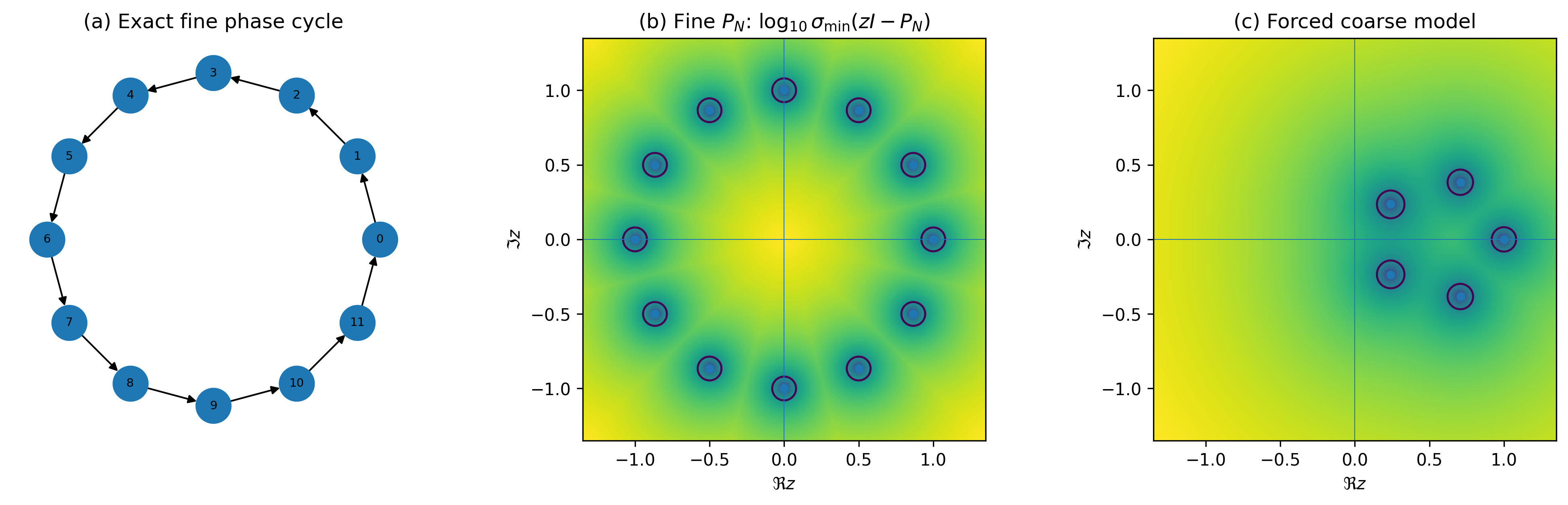}
\caption{\label{fig:rotation-pseudo}
Rotational calibration. (a) Exact fine phase cycle. (b) Resolvent diagnostic
for the normal cyclic permutation $P_N$; the contour marks a fixed
$\epsilon$-pseudospectral boundary and the points are eigenvalues.
(c) The corresponding forced coarse Markov representation after phase
aggregation. The comparison illustrates why pseudospectra are informative
about representation robustness but are not, by themselves, closure tests.}
\end{figure*}

\subsection{Chaotic logistic dynamics: closure, invariant law, and relaxation robustness}

Consider the fully chaotic logistic map~\cite{may1976,lasota1994}
\begin{equation}
x_{t+1}=f(x_t)=4x_t(1-x_t),
\qquad
x_t\in[0,1].
\label{eq:logistic-map}
\end{equation}
The scalar state is exactly first-order deterministic, but a finite localization
need not be. If $H_m$ denotes a cell map, then points can satisfy
\begin{equation}
H_m(x_1)=H_m(x_2),
\qquad
H_m(f(x_1))\neq H_m(f(x_2)),
\label{eq:logistic-obstruction}
\end{equation}
so apparent memory may be induced entirely by the observation. This separates
intrinsic chaotic complexity from representation-induced non-closure.

Probability trajectories were generated from ensembles of nonstationary initial
laws and identified with graph-constrained column-stochastic SSRC matrices.
For uniform partitions, the one-step total-variation error decreases from
approximately $0.0229$ at $m=4$ to $0.0158$ at $m=16$, then rises again to
approximately $0.0193$ at $m=32$. Meanwhile, the smallest nonzero singular
value of the empirical design decreases by nearly two orders of magnitude.
Thus finer localization does not produce a monotone improvement in
identifiability.

At fixed $m=12$, observable delay improves one-step prediction from
\begin{equation}
E_{\rm TV}(0)\approx0.01592
\end{equation}
to a shallow minimum near
\begin{equation}
E_{\rm TV}(4)\approx0.01129.
\end{equation}
A high-accuracy re-estimation gives
$E_{\rm TV}(3)\approx0.011411$,
$E_{\rm TV}(4)\approx0.011286$,
$E_{\rm TV}(5)\approx0.011347$, and
$E_{\rm TV}(6)\approx0.011539$. Over the same range,
$\sigma_{\min}^{+}$ decreases from approximately $0.1473$ to $0.00280$.
An instantaneous quadratic stochastic lifting also improves one-step
prediction, but its empirical-span conditioning is substantially poorer.
The chaotic example therefore exhibits the same basic distinction as the
general theory: nonlinear richness and observable history can both reduce
approximation error, but through different mechanisms and with different
identifiability costs.

\subsubsection{Invariant law as a representation diagnostic}

For $f(x)=4x(1-x)$, the invariant density is known explicitly~\cite{lasota1994},
\begin{equation}
\rho_\infty(x)
=
\frac{1}{\pi\sqrt{x(1-x)}},
\label{eq:logistic-invariant-density}
\end{equation}
with cumulative distribution
\begin{equation}
F_\infty(x)
=
\frac{2}{\pi}\arcsin\sqrt{x}.
\end{equation}
This gives exact invariant cell masses
\begin{equation}
p_j^\star
=
\int_{S_j}\rho_\infty(x)\,dx,
\end{equation}
providing an analytic benchmark that is independent of short-horizon
prediction.

The invariant law can also be used to design the partition itself. Equal
invariant-mass cells have boundaries
\begin{equation}
b_j
=
\sin^2\left(
\frac{\pi j}{2m}
\right),
\qquad
j=0,\ldots,m,
\label{eq:invariant-mass-partition}
\end{equation}
so that $p_j^\star=1/m$. Uniform and invariant-mass partitions do not rank
identically under validation error, stationary-law error, and conditioning.
This motivates distinguishing
\begin{equation}
\boxed{
\text{predictive fidelity}
\qquad\text{from}\qquad
\text{ergodic fidelity}.
}
\label{eq:predictive-ergodic}
\end{equation}

A useful stationary finite-state reference is
\begin{equation}
W^\star_{ij}
=
\frac{
\displaystyle
\int_{S_j}
\mathbf 1_{S_i}(f(x))\rho_\infty(x)\,dx
}{
\displaystyle
\int_{S_j}
\rho_\infty(x)\,dx
},
\label{eq:stationary-reference}
\end{equation}
for which $W^\star p^\star=p^\star$. This matrix is an
invariant-measure-conditioned coarse reference, not an exact closure operator
for arbitrary nonstationary laws. Accordingly,
$\|\widehat W-W^\star\|_2$ measures discrepancy from a stationary reference
rather than pure statistical identification error.

\begin{figure*}[t]
\centering
\includegraphics[width=0.96\textwidth]{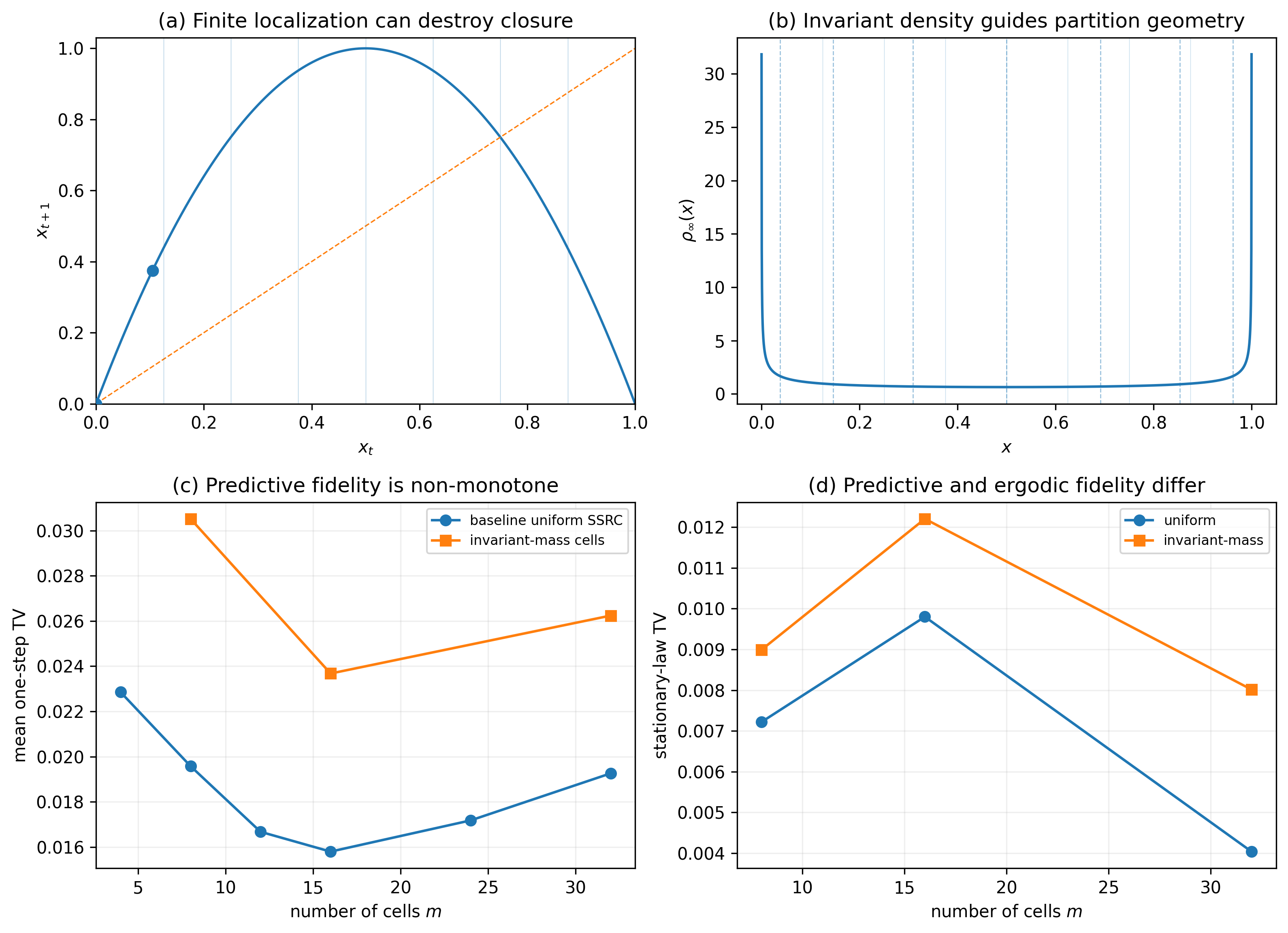}
\caption{\label{fig:logistic-invariant}
Chaotic logistic-map localization and invariant-law diagnostics.
(a) Finite cells can merge points with different next-cell images.
(b) The exact invariant density suggests a nonuniform equal-mass partition.
(c) Short-horizon predictive error is non-monotone in spatial resolution and
depends on partition geometry.
(d) Stationary-law fidelity need not rank representations in the same order as
one-step prediction.}
\end{figure*}

\subsubsection{Probability-relaxation pseudospectra}

For a column-stochastic matrix $W$, the zero-mass subspace
\begin{equation}
\mathcal Z
=
\left\{
v\in\mathbb R^m:
\mathbf 1^\top v=0
\right\}
\label{eq:zero-mass-subspace}
\end{equation}
is invariant. Since differences between probability vectors lie in
$\mathcal Z$, the nontrivial relaxation dynamics is more naturally studied
after removing the stationary mode. If the columns of
$Q\in\mathbb R^{m\times(m-1)}$ form an orthonormal basis of $\mathcal Z$, define
\begin{equation}
A
=
Q^\top WQ.
\label{eq:zero-sum-restriction}
\end{equation}
The spectrum of $A$ describes asymptotic relaxation, while
$\Lambda_\epsilon(A)$ and
\begin{equation}
G(k)
=
\|A^k\|_2
\label{eq:transient-growth}
\end{equation}
quantify perturbation sensitivity and transient probability amplification.

The logistic experiments reveal a substantial distinction between asymptotic
and transient behavior. For uniform partitions, the maximum transient
amplification increases from approximately $1.29$ at $m=8$ to $2.14$ at
$m=32$, despite spectral radii on the zero-mass subspace remaining below one.
Under equal invariant-mass partitions, the corresponding maxima remain close
to one over the same resolutions. Thus
\begin{equation}
\rho(A)<1
\quad\not\Rightarrow\quad
\|A^k\|_2\le1
\end{equation}
for all $k$, and the geometry of the finite representation can alter
transient robustness even when the underlying map is unchanged.

This comparison also clarifies the role of pseudospectra in the present
framework. Closure defect, non-normality, transient amplification, and
pseudospectral inflation are related but distinct. Pseudospectral contours are
therefore interpreted as robustness diagnostics of an already identified
square stochastic representation, not as direct certificates of closure.

\begin{figure*}[t]
\centering
\includegraphics[width=0.96\textwidth]{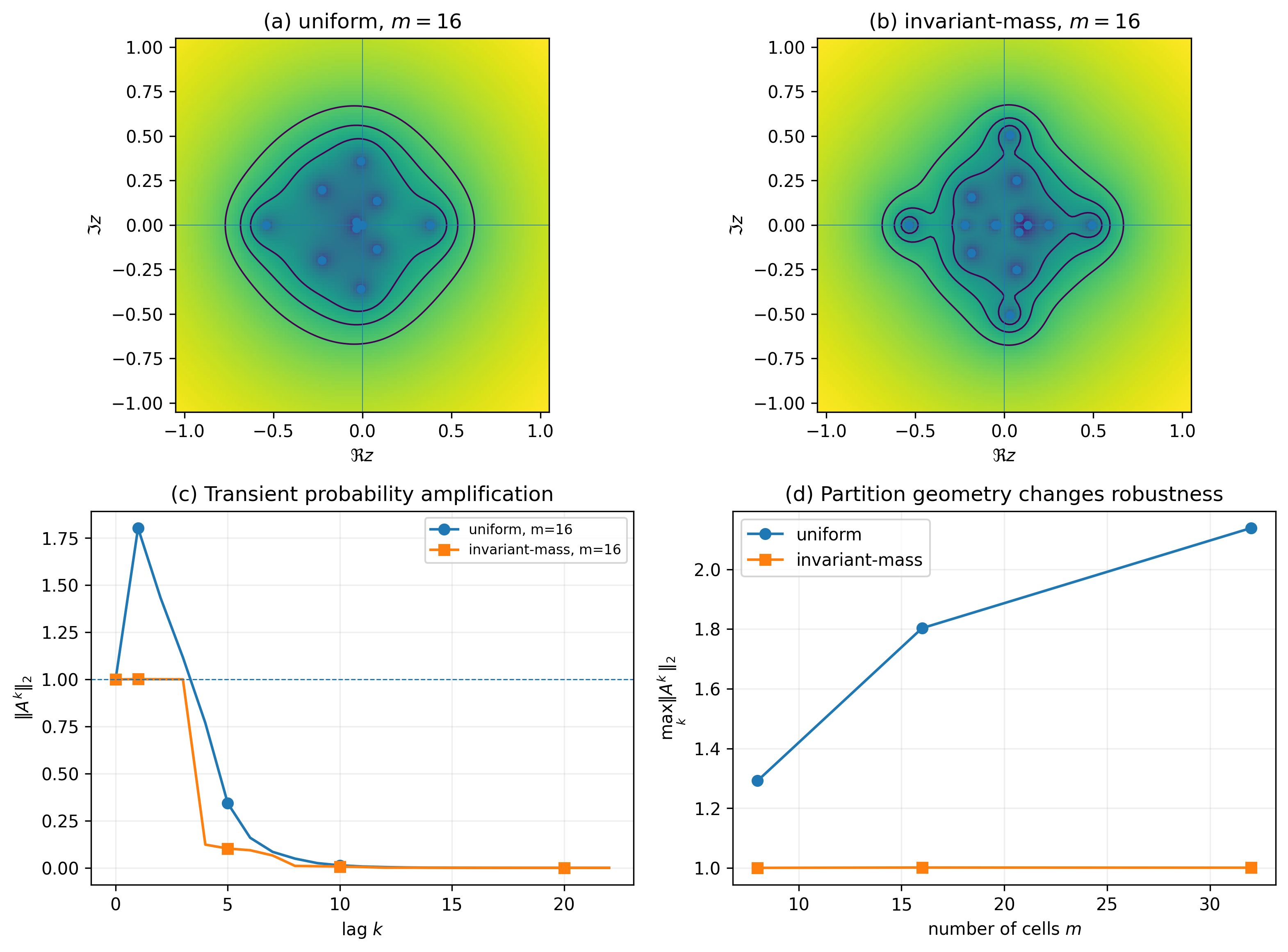}
\caption{\label{fig:logistic-pseudo}
Probability-relaxation robustness for the chaotic logistic map.
(a,b) Pseudospectral portraits of the zero-mass restriction $A=Q^\top WQ$
for uniform and invariant-mass partitions at $m=16$.
(c) Transient amplification $\|A^k\|_2$ after removing the stationary mode.
(d) The maximum amplification depends strongly on partition geometry:
uniform refinement increases transient growth, whereas equal invariant-mass
partitions remain close to contractive in the tested resolutions.}
\end{figure*}

\subsection{Van der Pol localization and adaptive spatial-temporal scale}

For the Van der Pol oscillator~\cite{vanderpol1926}
\begin{equation}
\dot x_1=x_2,
\qquad
\dot x_2=\mu(1-x_1^2)x_2-x_1,
\end{equation}
finite representative sets on the attracting cycle define empirical
Voronoi cells~\cite{okabe2000} and a relational graph. Varying the number of representatives
changes the empirical covering radius $\delta_m$, residence statistics, and
effective graph support.

The laboratory therefore searches jointly over $(m,\alpha)$. Candidate
models are compared by validation error, support complexity, and spatial
resolution. Rather than asserting a unique optimum, an admissible region
\begin{equation}
\mathcal R_\tau
=
\left\{
(m,\alpha):
E_{\rm val}(m,\alpha)
\le
(1+\tau)E_{\min}
\right\}
\label{eq:admissible-region}
\end{equation}
identifies statistically comparable representations.

At a fixed spatial-temporal representation, polynomial degree and delay
depth were compared. The quadratic model $(p,r)=(2,1)$ achieved lower
validation error than all tested linear models through $r=4$, while the
linear model at $r=5$ overtook it and the extended linear sweep continued to
improve. This supports the distinction between nonlinear representational
richness and genuinely informative history.

\begin{figure*}[t]
\centering
\includegraphics[width=0.96\textwidth]{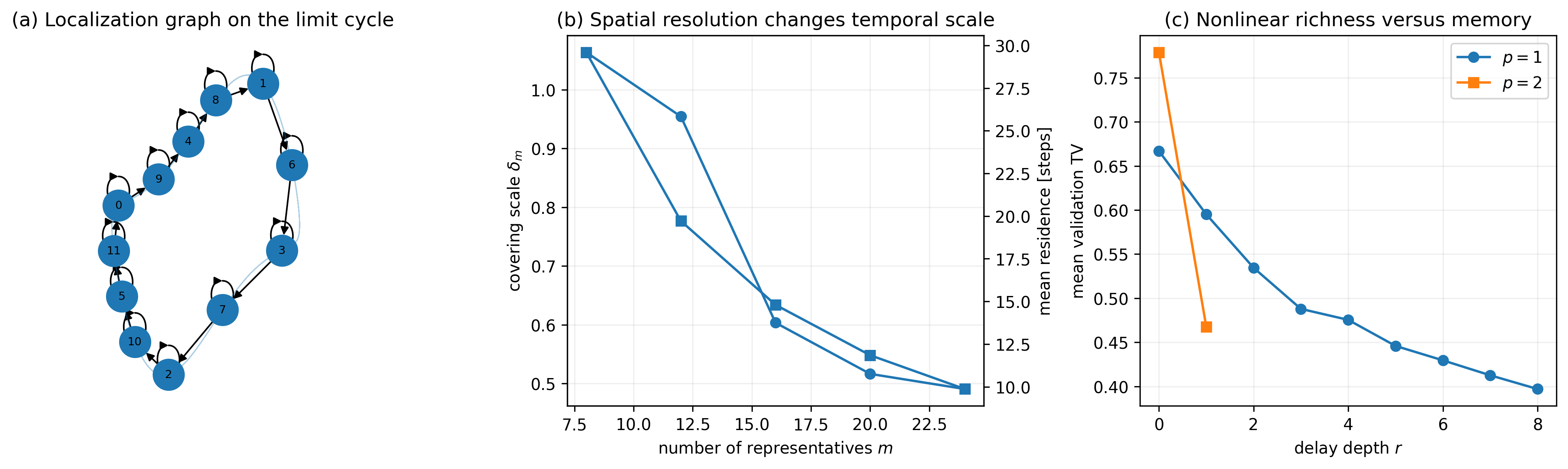}
\caption{\label{fig:vdp-tradeoffs}
Van der Pol localization and representation trade-offs. (a) A finite
relational graph placed on the attracting limit cycle. (b) Increasing the
number of representatives decreases the covering scale while also changing
residence time, coupling spatial and temporal resolution. (c) At the
selected spatial-temporal scale, a quadratic model with one delay is more
accurate than linear models through $r=4$, but sufficiently long linear
history eventually performs better.}
\end{figure*}

\subsection{Synthetic cyclic inventory: hidden age and approximate demand cycles}

The operational inventory states are
\begin{equation}
\mathrm{High}
\to
\mathrm{Normal}
\to
\mathrm{Low}
\to
\mathrm{Reorder}
\to
\mathrm{Transit}
\to
\mathrm{Restocked}.
\end{equation}
The fine model resolves transit progress as
\begin{equation}
T_1\to T_2\to\cdots\to T_L,
\end{equation}
whereas the coarse observation merges all $T_a$ into
$\mathrm{Transit}$. Hence
\begin{equation}
He_{T_1}=He_{T_L},
\qquad
HW_fe_{T_1}\neq HW_fe_{T_L}.
\end{equation}
The structured fine SSRC model recovers the synthetic transition matrix to
machine precision, while the coarse representation has a nonzero exact-data
residual.

Observable-memory models improve sharply when the delay horizon approaches
the physical lead-time scale. More importantly, the coarse observed paths
alone can trigger state enrichment. The empirical hazard
\begin{equation}
\widehat h_j(a)
=
\Pr(s_{t+1}\neq j\mid s_t=j,A_t=a)
\end{equation}
is tested against a constant-hazard null. In the synthetic experiment only
the transit state is flagged. It is then refined automatically into
age-indexed states without using the hidden labels $T_1,\ldots,T_L$.
Held-out transition prediction improves from log loss approximately
$0.552$ to $0.419$ and from Brier score approximately $0.375$ to $0.284$.
Thus the enrichment is detected statistically and accepted by predictive
validation. As a complementary robustness check, the age-enriched matrix
also has a smaller departure from normality,
$\|\What^\top\What-\What\What^\top\|_F$, than the coarse matrix
($1.25$ versus $1.89$), consistent with the more favorably conditioned
pseudospectrum shown in Fig.~\ref{fig:inventory-pseudo}.

\begin{figure*}[t]
\centering
\includegraphics[width=0.8\textwidth]{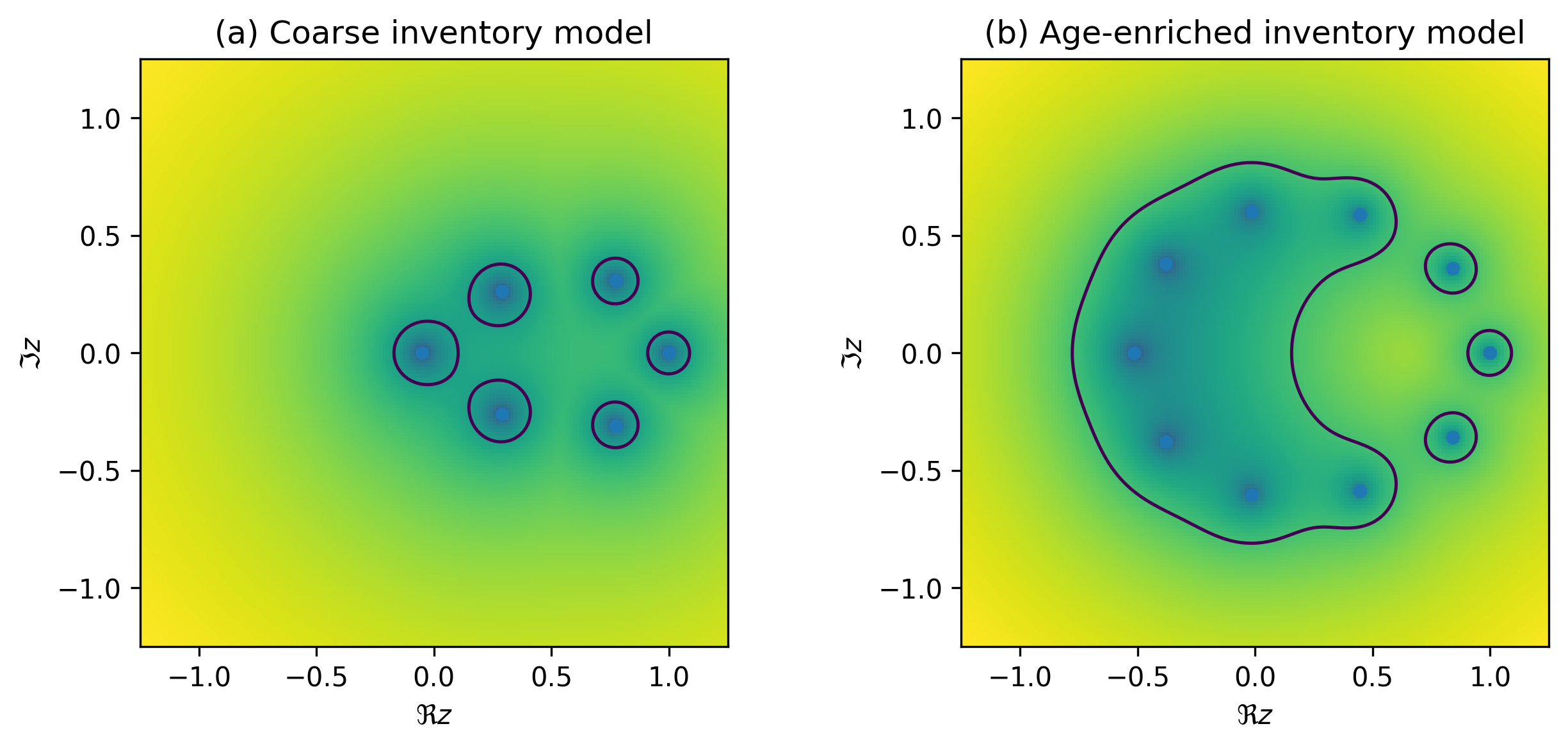}
\caption{\label{fig:inventory-pseudo}
Pseudospectral portraits of the inventory transition matrix before and
after age enrichment, at a common perturbation scale $\epsilon=0.08$.
(a) The coarse model, restricted to the operational states of
Section~\ref{sec:example}, has isolated near-unity eigenvalues surrounded
by tight resolvent contours. (b) The age-enriched model has a visibly
larger and more merged $\epsilon$-pseudospectral region around its
dominant cluster, together with a smaller departure from normality overall,
consistent with the improved held-out predictive closure reported above.
This is a second instance of the same diagnostic role played by
Fig.~\ref{fig:rotation-pseudo}: pseudospectra are used here as a
robustness comparison between two already-identified square
representations, not as a certificate of closure.}
\end{figure*}

\begin{figure*}[t]
\centering
\includegraphics[width=0.96\textwidth]{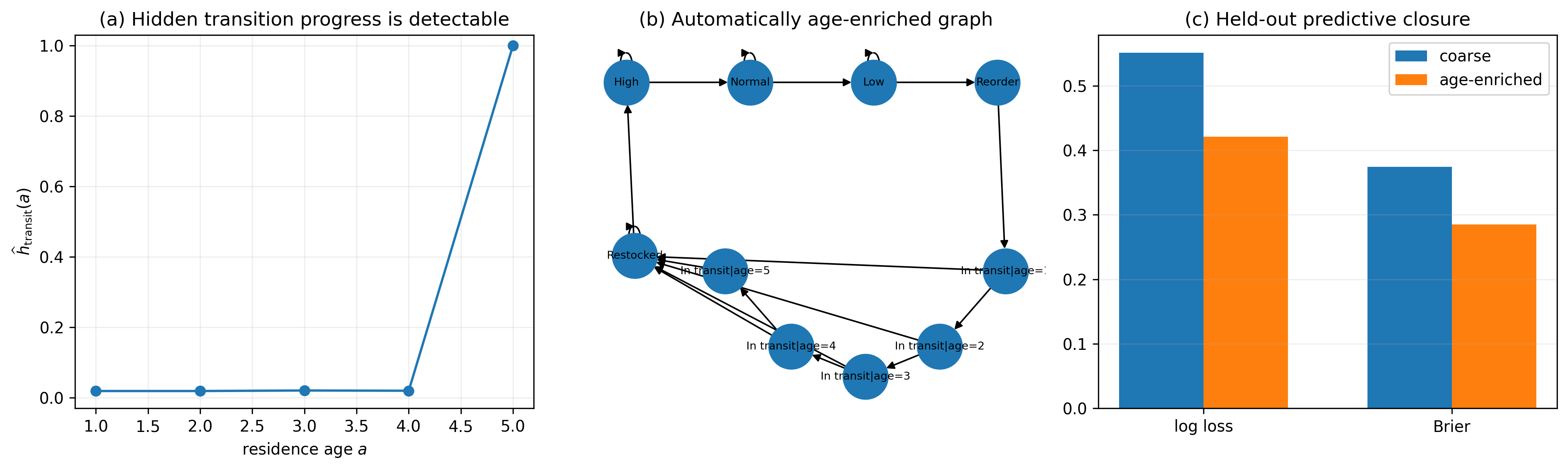}
\caption{\label{fig:inventory-enrichment}
Algorithmic mechanistic enrichment in the synthetic inventory cycle.
(a) The exit hazard from the coarse transit state depends strongly on
elapsed residence age. (b) The flagged state is refined automatically into
age-indexed states using only coarse paths. (c) The enriched representation
improves held-out log loss and Brier score.}
\end{figure*}

Finally, an approximately cyclic demand signal
\begin{equation}
u_t
=
\bar u
+
a_1\sin(2\pi t/T_1+\phi_1)
+
a_2\sin(2\pi t/T_2+\phi_2)
\end{equation}
modulates stock-depletion probabilities. If $u_t$ is omitted, delay models
must infer part of its phase indirectly from inventory history. Appending
the observed forcing and using a quadratic stochastic lifting improves
recursive prediction relative to the hidden-demand delay sweep. This gives a
controlled example of
\begin{equation}
\text{omitted exogenous state}
\quad\Longrightarrow\quad
\text{apparent memory}.
\end{equation}
Component-wise residual dependence on $u_t$ provides an algorithmic trigger
for testing exogenous-state enrichment.

\begin{figure}[t]
\centering
\includegraphics[width=\columnwidth]{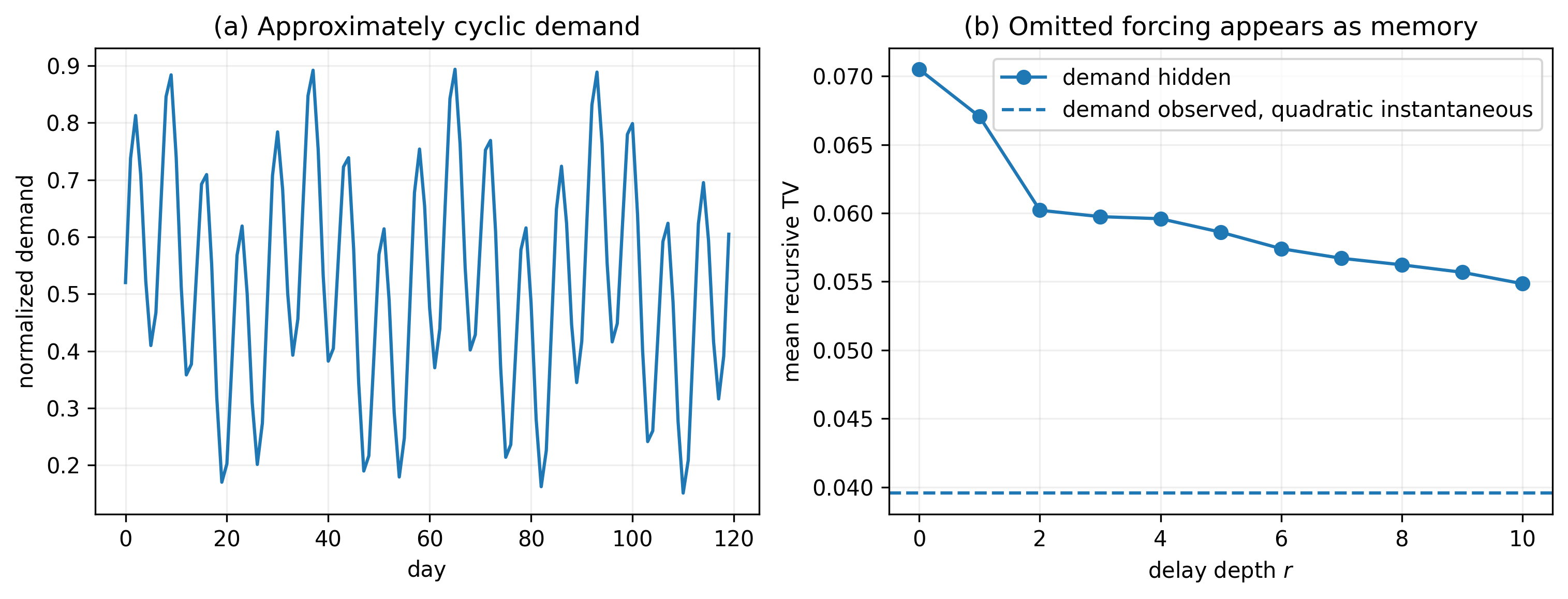}
\caption{\label{fig:demand-memory}
Approximate demand cycles and apparent memory. The exogenous forcing contains
a weekly and a slower component. When demand is omitted, delay depth partly
reconstructs its phase from inventory history; supplying the observed
forcing directly lowers recursive prediction error without requiring a
long delay.}
\end{figure}

\subsection{Common structure of the examples}

The route pending edge, rotational phase, within-cell logistic position,
fine Van der Pol localization, inventory transit age, and demand phase are
different physical or operational objects, but they expose the same
representational question: which distinctions are required for prediction?
The logistic example is especially useful because the original scalar state
is already Markovian; the apparent memory is introduced by finite
localization itself. Delay lifting provides a reconstructive alternative
whenever finite histories separate the unresolved distinctions.

\section{\label{sec:selection}Adaptive representation selection}

The numerical experiments indicate that representation complexity is organized
along interacting axes:
\begin{equation}
\boxed{
\delta,\qquad
\alpha,\qquad
p,\qquad
r,
}
\label{eq:adaptive-axes}
\end{equation}
together with graph support and optional state-space enrichment.

Here $\mathcal P$ controls spatial localization, including both scale
$\delta$ and partition geometry; $\alpha$ controls temporal sampling,
$p$ instantaneous nonlinear richness, and $r$ observable memory depth. State
enrichment changes the informational content of the representation and
therefore cannot be reduced to tuning $p$ or $r$.

\subsection{Minimal adequate representations}

Let $\mathcal R$ denote a candidate representation and let
$E_{\rm val}(\mathcal R)$ be an out-of-sample predictive error. A natural
goal is
\begin{equation}
\min_{\mathcal R}
\mathcal C(\mathcal R)
\quad\text{subject to}\quad
E_{\rm val}(\mathcal R)\le\epsilon_{\rm target},
\label{eq:min-adequate}
\end{equation}
with additional graph, conditioning, or measurement constraints. This
defines a \emph{minimal adequate representation}: not the finest available
model, but the least complex one that achieves the required closure.

When no single tolerance is preferred, a Pareto frontier can be formed from
validation error and complexity measures such as
\begin{equation}
\operatorname{nnz}(\widehat W),
\qquad
\dim(\mathcal R),
\qquad
\delta,
\end{equation}
or measurement cost. For square candidate representations, one may also
include a pseudospectral robustness constraint or Pareto coordinate, for
example $r_\epsilon(A)$ at a common perturbation scale. This criterion is
applied only after closure and representation adequacy have been assessed.
An admissible region such as Eq.~\eqref{eq:admissible-region} is often more
defensible than a unique minimizer.

\subsection{Conditioning and empirical span}

For lifted data matrix $Y_{\mathcal R}$, residual reduction alone is
insufficient. Simplex and delay coordinates can be structurally rank
deficient, so conditioning should be measured on the empirical span using
the smallest nonzero singular value,
\begin{equation}
\sigma_{\min}^{+}(Y_{\mathcal R}).
\end{equation}
A richer representation can improve predictive closure while worsening
conditioning or sample complexity. This tradeoff is intrinsic.

\subsection{Closure-driven greedy enrichment}

A practical adaptive procedure is:
\begin{enumerate}
\item fit the structured stochastic model on the current representation;
\item evaluate held-out error and residual structure;
\item generate candidate changes suggested by the diagnostics:
spatial refinement, temporal resampling, small delays, polynomial lifting,
residence-age refinement, or exogenous variables;
\item validate each candidate under the same predictive criterion;
\item accept the candidate with sufficient gain relative to complexity and
measurement cost;
\item stop when the target error is met or no candidate gives a meaningful
improvement.
\end{enumerate}

This procedure is intentionally conservative. A periodic residual can trigger
a candidate phase variable; age-dependent hazard can trigger residence
refinement; residual correlation with $u_t$ can trigger exogenous
enrichment; and a residual localized in a subset of cells can trigger local
spatial refinement. The algorithm proposes predictive refinements, while
domain knowledge determines whether they have a credible mechanistic
interpretation.

\section{\label{sec:discussion}Discussion}

The resulting framework may be viewed as a theory of finite stochastic
representations whose geometry, clock, memory, and information content are
adapted to the observable dynamics.

The first structural distinction is between \emph{representation} and
\emph{information}. An SSRC lifting with active linear block satisfies
\begin{equation}
\kappa_p\circ\eth_p=\operatorname{id},
\end{equation}
so it enriches coordinates without discarding or adding observable
information. Delay lifting adds information from observable history.
State-space enrichment refines an information-losing observation map.

This can be summarized by
\begin{equation}
\widetilde{\mathcal X}
\overset{\widetilde H}{\longrightarrow}
\mathcal Z
\overset{\Pi}{\longrightarrow}
\mathcal X
\overset{\eth_p}{\longrightarrow}
\mathcal M_p,
\end{equation}
where $H=\Pi\circ\widetilde H$. The refinement
$\widetilde H$ aims to make future observables approximately constant on its
fibers; the stochastic lifting $\eth_p$ then provides a structured
representation for identification.

The examples reveal several recurring mechanisms. Route commitment and
inventory transit age are hidden transition progress. Rotational phase and
spatial localization are geometric state distinctions. The logistic map
shows that finite localization can itself manufacture an apparent memory
requirement even when the underlying state is exactly Markovian. Demand
phase is an exogenous distinction. Across these settings, the central
failure is the same: states or probability configurations merged by the
current representation can have materially different observable futures.

A useful implication follows for memory. A selected delay depth need not
indicate intrinsic non-Markovianity. It can arise because a mechanistic
state or exogenous forcing has been omitted, or because finite localization
has discarded within-cell information that is partially reconstructible from
history. Adding the right state variable, refining the localization, or
retaining a short observable history can therefore reduce apparent memory.

\subsection{A universality question}

Let
\begin{equation}
F:
\mathcal K
\subset
(\Delta^{d-1})^{r+1}
\to
\Delta^{d-1}
\end{equation}
be a continuous probability update on a compact family of admissible
histories. Under what conditions do there exist $p$, $r$, a stochastic
lifting $\Epr$, and a structured column-stochastic matrix $W_{p,r}$ such
that
\begin{equation}
\sup_{\mathbf p\in\mathcal K}
\left\|
F(\mathbf p)-W_{p,r}\Epr(\mathbf p)
\right\|
<\epsilon?
\label{eq:universality-question}
\end{equation}
No such universality theorem is claimed here. The question is constrained
by simplex geometry, nonnegativity, graph support, and possibly incomplete
information.

A second approximation problem concerns enrichment itself: given a coarse
observation $H$, can one construct a low-complexity refinement
$\widetilde H$ whose fibers are $K$-step predictively sufficient to within a
prescribed tolerance? This is a finite predictive-quotient problem rather
than full state reconstruction.

\subsection{Open directions}

Several directions follow directly.

First, hazard-driven refinement should be generalized beyond residence age
to other local statistics of hidden progress. Second, exogenous enrichment
should be extended to controlled or switching structured stochastic models.
Third, local temporal scales $\alpha(\delta,j)$ may be useful when residence
times vary strongly by region. Fourth, conditioning and sample complexity
should be incorporated explicitly into adaptive selection. Fifth, spectral
and pseudospectral analysis for nonlinear lifted dynamics should proceed
through Jacobians, cocycles, or related local objects, not through the
rectangular readout matrix.

Finally, the common structure of the examples suggests a broader viewpoint:
probability localization models are not only reduced models of dynamics.
They are adaptive quotients of the underlying state, refined until their
observable fibers become sufficiently predictive for the task at hand.
The quotient/refinement language already exposes a topological layer, but
more specific questions about preserved cycles, homology, or topological
obstructions are deliberately left for separate work.

\begin{acknowledgments}
The author acknowledges the Department of Applied Mathematics, School of
Mathematics and Computer Science, UNAH, for institutional support. The
author also thanks Ruth Moreno for insightful conversations that motivated
the worked route-network example.
\end{acknowledgments}

\section*{Data Availability Statement}

The numerical examples in this work are synthetic and are generated by the
algorithms described in the manuscript and accompanying computational
notebooks. No external empirical data are required to reproduce the reported
experiments. The code and computational notebooks that support the findings
and experiments reported in this manuscript will be made available, in due
time, in the ProSpectLifter GitHub
repository~\cite{vides2026prospectlifter}.

\appendix

\section{\label{app:reductions}Technical details}

\subsection{Reduction of repeated monomials}

For a tensor block $x^{\otimes k}$, associate each coordinate
$x_{i_1}\cdots x_{i_k}$ with the exponent multiindex
\begin{equation}
\alpha_j
=
\#\{
\ell:i_\ell=j
\}.
\end{equation}
Two tensor coordinates represent the same commutative monomial if and only
if their exponent multiindices coincide. The reduction matrix $R$ contains
one row per distinct multiindex and sums all tensor coordinates belonging to
that class. Since only coordinates are aggregated and no mass is removed,
the reduced vector remains stochastic.

For a base dimension $m$ and quadratic degree, the unreduced linear-plus-
quadratic embedding has dimension
\begin{equation}
m+m^2,
\end{equation}
whereas the reduced dimension is
\begin{equation}
m+\binom{m+1}{2}.
\end{equation}
For the delay-two route example, $m=18$, so the dimension decreases from
\begin{equation}
342
\end{equation}
to
\begin{equation}
189.
\end{equation}

\subsection{Left inverse after reduction}

Let $\Pi_1$ select the reduced first-order coordinates. If the linear block
has coefficient $\alpha_1>0$, then
\begin{equation}
\Pi_1\eth_p(x)
=
\alpha_1x
\end{equation}
and
\begin{equation}
\kappa_p(q)
=
\alpha_1^{-1}\Pi_1q.
\end{equation}

For the delay embedding of Eq.~\eqref{eq:delay-block}, the reduced linear
coordinates contain
\begin{equation}
\alpha_1v_0p_t,
\ldots,
\alpha_1v_rp_{t-r}.
\end{equation}
If every $v_\ell>0$, blockwise rescaling reconstructs the complete history.

\subsection{Residual-to-matrix bound}

Equation~\eqref{eq:matrix-bound} follows directly from
\begin{equation}
(\What-\Wtrue)Y
=
(\What Y-Y')+E
\end{equation}
and right multiplication by the Moore--Penrose inverse $Y^\dagger$ when $Y$
has full row rank.

\section*{References}

\bibliographystyle{aipnum4-1}
\bibliography{paper/probability_localization_dynamics}

\end{document}